\documentclass[two side,reqno,11pt]{amsart}
\usepackage{amssymb,latexsym,xy,eucal,mathrsfs}
\usepackage{amssymb,amsmath,graphicx,colortbl,enumerate}
\usepackage{blindtext}
\usepackage{lineno}
\usepackage{lipsum}
\usepackage{tikz}
\usetikzlibrary{decorations.markings}
\usepackage{pst-plot}
\usepackage{rotating}
\usepackage{pgfplots}
\usepackage{amsmath}
\usepackage{longtable,mathrsfs}
\usepackage[labelfont=bf,justification=raggedright,singlelinecheck=false]{caption}
\usepackage{caption}
\usepackage{subcaption}
\usepackage{framed}
\usepackage{float}

\numberwithin{equation}{section}
\newtheorem{thm}{Theorem}[section]

\newtheorem{theorem}[thm]{Theorem}

\newtheorem{example}[thm]{Example}

\allowdisplaybreaks
\DeclareMathSizes{10}{9}{5}{5}   
\usepackage{mathtools}

\begin{document}
	
	\baselineskip 17truept
	\title{ Limit cycles of piecewise smooth differential systems of the type nonlinear center and saddle }
	\date{}
	\author{Nanasaheb Phatangare, Krishnat Masalkar and Subhash Kendre}
	\address{Nanasaheb Phatangare, Department of Mathematics, Fergusson College, Pune,
		-411004, India.(M.S.)} \email{nmphatangare@gmail.com}
	\address{Krishnat Masalkar, Department~of~Mathematics, Abasaheb ~Garware~ College, Pune-411004, India
		(M.S.)} \email{krishnatmasalkar@gmail.com}
	
	\address{Subhash Kendre, Department~of~Mathematics, Savitribai~Phule~Pune~University, Pune-411007, India
		(M.S.)} \email{sdkendre@yahoo.com}
	
	\makeatletter
	\@namedef{subjclassname@2020}{%
		\textup{2020} Mathematics Subject Classification}
	\makeatother
	
	
	\maketitle \noindent{ \small \textbf{Keywords}:  Piecewise linear differential system, Piecewise smooth differential system, limit cycle,  Hamiltonian, first integral, level curves, resultant.
	}
	
	\begin{abstract}
		Piecewise linear differential systems separated by two parallel straight lines of the type of center-center-Hamiltonian saddle and the center-Hamiltonian saddle-Hamiltonian saddle can have at most one limit cycle and there are systems in these classes having one limit cycle. In this paper, we study the limit cycles of a piecewise smooth differential system separated by two parallel straight lines formed by nonlinear centers and a Hamiltonian saddle. 
	\end{abstract}
	\section { Introduction} 
	The dynamical system is a powerful tool for understanding a diverse range of problems. There are now well-developed qualitative techniques to study smooth dynamical systems \cite{bernardo2008piecewise,meiss2007differential}. It has been proven an extremely effective tool for understanding the behaviour of many physical phenomena such as elastic deformation, nonlinear optical, fluid flows, and biological systems. \par Many dynamical systems that occur naturally are characterized by periods of smooth evolutions affected by instantaneous events. Such dynamical systems are called piecewise smooth dynamical systems. Recently, non-smooth processes such as impact, switching, sliding, and other discrete state transitions have been studied widely. These phenomena arise in any application involving friction, collision, intermittently constrained systems, and processes with switching components. Recently, literature has drawn attention towards nonsmooth dynamical systems, for instance, see \cite{brogliato2000impacts,brogliato1999nonsmooth,kunze2001non}. The study of limit cycles is one of the effective techniques for understanding and analyzing smooth and non-smooth dynamical systems. The limit cycle is an isolated periodic orbit of a dynamical system. In the literature, the study of limit cycles goes back to the nineteenth century, for instance, see \cite{poincare1891integration,andronov1966theory}. Many phenomena such as the Belousov-Zhabotinskii reaction, and Van der Pol oscillator shows the existence of limit cycles, for details, see \cite{van1920theory,van1926lxxxviii,belusov1959periodically,zhabotinsky1964periodical}. 
	
	A continuous piecewise linear differential system (PLDS) formed by two different systems that are separated by a straight line has at most one limit cycle, for instance, see \cite{freire1998bifurcation,buzzi2013piecewise,lum1991global,lum1992global}.
	\par The limit cycle can lie in one zone or it can reside in two or more zones when it comes to piecewise linear differential systems. Sliding limit cycles or crossing limit cycles are limit cycles that span several zones. In this paper, we discuss the crossing limit cycles of piecewise smooth nonlinear differential systems.
	\par It is well known that a discontinuous PLDS can have three limit cycles when they are separated by a straight line, for details, see \cite{buzzi2013piecewise,de2013limit,cardoso2020simultaneous,freire2014general,huan2012number,li2014three}.
	\par There is no limit cycle in a continuous piecewise differential system with two linear centers separated by a straight line. Also, there is no limit cycle in a discontinuous PLDS formed by two linear centers and separated by a straight line, but there are systems with one limit cycle when the system is formed by three linear centers and separated by two straight lines, for instance, see \cite{llibre2018piecewise}.
	\par In \cite{llibre2021limit}, it is proved that the limit cycles are not present in continuous or discontinuous PLDS consisting of two linear Hamiltonian saddles separated by a straight line. A  continuous PLDS formed by linear saddles in three zones separated by two straight lines has one limit cycle.
	\par In \cite{llibre2022limit}, it has been proven that there are no limit cycles in a continuous or discontinuous PLDS formed by a center and a Hamiltonian saddle and separated by a single straight line. The results from \cite{llibre2022limit} were also obtained in \cite{pessoa2022limit}, in fact, this paper contains all the results mentioned here from the papers \cite{llibre2018piecewise,llibre2021limit,llibre2022limit}.
	\par In \cite{franccoise2021quadratic}, we find a detailed discussion of a quadratic planar vector field with a double center and its normal form. In this paper, we use this normal form to study the limit cycles of planar piecewise differential systems placed in two and three zones. In \cite{yang2020limit}, a planar integrable quadratic vector field with the global center is studied for its limit cycle bifurcation. We use it to study the limit cycles of a piecewise planar vector field placed in three zones.
	\par In this paper, the bounds for the number of limit cycles of continuous piecewise differential systems (PDS) and discontinuous piecewise differential systems in two zones and systems in three zones are discussed. The paper is organized as follows: In \textit{Section} \ref{Ch4sec2}, piecewise smooth systems placed in two and three zones formed by a quadratic center and a linear saddle separated by a straight line are discussed. \textit{Section} \ref{Ch4sec3} is to discuss the limit cycles of piecewise smooth systems formed by the general quadratic center and linear saddle. \textit{Section} \ref{Ch4sec4} presents some applications of the results obtained in \textit{Section \ref{Ch4sec2}} and \textit{Section \ref{Ch4sec3}}.
	\section{Quadratic center and linear saddle}\label{Ch4sec2}
	Consider the quadratic polynomial differential  system having a linear center at $(0,0)$;
	\begin{align}\label{Ch4s2e_1}
		\dot{X}=&\begin{cases}
			-y+a_1x^2+b_1xy+c_1y^2\\
			x+a_2x^2+b_2xy+c_2y^2,
		\end{cases}
	\end{align}
	where $a_i,b_i,c_i,i=1,2,$ are parameters.
	
	Using the rotation,
	\begin{align*}
		u=&x\cos \theta -y \sin \theta ~;\\ v=&x\sin \theta+y\cos \theta,
	\end{align*}
	the system (\ref{Ch4s2e_1}) becomes, 
	\begin{align*}
		\begin{bmatrix} \dot{u}\\\dot{v} \end{bmatrix}&=\begin{bmatrix} \cos \theta &-\sin \theta \\ \sin \theta & \cos \theta \end{bmatrix} \begin{bmatrix} \sin \theta &-\cos \theta \\ \cos \theta & \sin \theta \end{bmatrix} \begin{bmatrix}u\\v\end{bmatrix}+\begin{bmatrix} lu^2+muv+nv^2\\pu^2+quv+rv^2 \end{bmatrix} \nonumber\\
		&=\begin{bmatrix}-v\\u \end{bmatrix}+\begin{bmatrix} lu^2+muv+nv^2\\pu^2+quv+rv^2 \end{bmatrix},
	\end{align*}
	where \begin{align*}r=&\sin \theta (a_1\sin^2 \theta +b_1\sin \theta \cos \theta +c_1 \cos^2 \theta)+\cos \theta (a_2\sin^2 \theta +b_2\sin \theta \cos \theta +c_2 \cos^2 \theta)\\
		=&a_1\sin^3 \theta +(b_1+a_2)\sin^2 \theta \cos \theta +(c_1+b_2)\cos^2\theta\sin\theta +c_2\cos^3 \theta,	
	\end{align*} and $l,m,n,p,q$ are constants depending on $\theta$.
	
	Since the expression for $r$ is a cubic homogeneous trigonometric polynomial, we can choose the value of $\theta$ such that $r=0$.
	
	Now renaming variables $u$ and $v$ by $x$ and $y$, respectively,
	we can write the system (\ref{Ch4s2e_1}) in the following normal form, 
	\begin{align}\label{Ch4s2e_2}
		\dot{X}=&\begin{cases}
			-y+lx^2+mxy+ny^2\\
			x+px^2+qxy.
		\end{cases}
	\end{align}
	In \cite{franccoise2021quadratic}, it is proved that the system (\ref{Ch4s2e_2}) is Hamiltonian  if $m=0, q=-2l,n\neq 0$ and the corresponding Hamiltonian is given by 
	\begin{align} \label{Ch4s2e_2.3}
		F=\frac{1}{2}(x^2+y^2)-lx^2y-\frac{n}{3}y^3+\frac{p}{3}x^3.
	\end{align}
	Moreover, the system has centers at $(0,0)$ and $(0,\frac{1}{n})$.\\
	Thus, the Hamiltonian system with a centers at $(0,0)$ and $(0,\frac{1}{n})$ is given by 
	\begin{align}\label{Ch4s2e2}
		\dot{X}=&\begin{cases}
			-y+lx^2+ny^2\\
			x+px^2-2lxy
		\end{cases}
	\end{align}
	with $n\neq0.$

\par	The quadratic integrable system having unique center at $(0,\xi)$, as in \cite{yang2020limit}, takes the normal form, 
	\begin{align}\label{Ch4s2e_3}
		\dot{X}=&\begin{cases}
			y-2x^2-\xi\\
			-2xy
		\end{cases},
	\end{align} where $\xi>0$.

	The first integral of the systems (\ref{Ch4s2e_3})  is $$I(x,y)=\dfrac{1}{y^2}\left(x^2-y+\dfrac{\xi}{2}\right).$$
	
	The normal form of a linear Hamiltonian system with the saddle at $$\left(-\dfrac{\beta \mu+\delta \gamma}{\alpha \delta -\beta^2},\dfrac{\alpha \mu+\beta \gamma}{\alpha \delta -\beta^2}  \right),$$ is given by
	\begin{align}\label{Ch4s2e_4}
		\dot{X}=\begin{cases}
			-\beta x-\delta y+\mu \\
			\alpha x+\beta y+\gamma,
		\end{cases} 
	\end{align} where $\alpha =0$ or $1$. Also, $\gamma=0, \beta \neq 0$ when $\alpha=0$ and  $\delta <\beta^2$ when $\alpha =1$, for instance, see \cite{llibre2021limit}.
	Note that the eigenvalues of system (\ref{Ch4s2e_4}) are $\pm \beta$ if $\alpha=0$ and $\pm \sqrt{\beta^2-\delta}$ if $\alpha=1$. Also, its saddle separatrices are given by
	\begin{align*}
		y-{\frac {\mu\,\alpha+\beta\,\gamma}{\delta\,\alpha-{\beta}^{2}}}+{
			\frac {\beta+\sqrt {-\delta\,\alpha+{\beta}^{2}}}{\delta} \left( x-{
				\frac {\beta\,\mu+\delta\,\gamma}{-\delta\,\alpha+{\beta}^{2}}}
			\right) }=&0,~ \text{and}\\
		y-{\frac {\mu\,\alpha+\beta\,\gamma}{\delta\,\alpha-{\beta}^{2}}}+{
			\frac {\beta-\sqrt {-\delta\,\alpha+{\beta}^{2}}}{\delta} \left( x-{
				\frac {\beta\,\mu+\delta\,\gamma}{-\delta\,\alpha+{\beta}^{2}}}
			\right) }=&0.
	\end{align*} 
	The Hamiltonian function for (\ref{Ch4s2e_4}) is
	\begin{align}\label{Ch4s2e_2.6}
		H(x,y)=&-\gamma x+\mu y-\beta xy-\dfrac{1}{2}(\alpha x^2+\delta y^2).
	\end{align}
	If $\beta=0$,  then the system (\ref{Ch4s2e_4})  becomes 
	\begin{align}\label{Ch4s2e_2.7}
		\dot{X}=\begin{cases}
			-\delta y+\mu \\
			x+\gamma.
		\end{cases} 
	\end{align}
	It is clear that the system (\ref{Ch4s2e_2.7}) has a saddle point at $\left(-\gamma,\dfrac{\mu}{\delta}\right)$ if  and only if $\delta<0$. Also, its Hamiltonian is given by
	$$H(x,y)=-\gamma x+\mu y-\frac{1}{2}(x^2+\delta y^2).$$
	Suppose that $\beta\neq 0$. Then by the rescaling of time $\tau=\beta t$, we can take $\beta=1$ in the system (\ref{Ch4s2e_4}).
	Thus, we can consider the general linear Hamiltonian system with the saddle as,
	\begin{align}\label{Ch4s2e_4.1}
		\dot{X}=\begin{cases}
			-x-\delta y+\mu \\
			\alpha x+ y+\gamma
		\end{cases} 
	\end{align}
	with $\alpha=0,~1$ and $\delta<1$. It has saddle separatrices 
	\begin{align*}
		y+\frac{\alpha\mu  +\gamma}{1-\alpha\delta }+\frac{1+\sqrt{1-\alpha \delta}}{\delta}\left( x-\frac{\mu+\delta \gamma}{1-\alpha\delta }\right)=&0,~ \text{and}\\
		y+\frac{\alpha\mu  +\gamma}{1-\alpha\delta }+\frac{1-\sqrt{1-\alpha \delta}}{\delta}\left( x-\frac{\mu+\delta \gamma}{1-\alpha\delta }\right)=&0.
	\end{align*} with Hamiltonian
	\begin{align} \label{Ch4s2e_2.10}
		H(x,y)=-\frac{1}{2}\alpha x^2-\frac{1}{2}\delta y^2-xy-\gamma x+\mu y.
	\end{align}
	In this paper, we use the double center (\ref{Ch4s2e2}) and the global center (\ref{Ch4s2e_3}) to study the limit cycles of planar piecewise differential systems.
	\subsection{Quadratic double center and linear saddle}
	Consider a piecewise differential system consisting of a nonlinear center and a linear saddle separated by a straight line;
	\begin{align}\label{Ch4s2e_6}
		\dot{X}=\begin{cases}
			(F_y(x,y),-F_x(x,y))&~\text{if}~x<0\\
			(H_y(x,y),-H_x(x,y))&~\text{if}~x>0,
		\end{cases}
	\end{align} where $F$ and $H$ are given by (\ref{Ch4s2e_2.3}) and (\ref{Ch4s2e_2.6}), respectively.
	The piecewise-smooth system (\ref{Ch4s2e_6}) is said to be continuous if, on the separation boundary, the extensions of the left-hand and right-hand vector fields coincide.
	\begin{theorem}
		Consider the PDS (\ref{Ch4s2e_6}). 
		\begin{enumerate}
			\item The system (\ref{Ch4s2e_6}) has no limit cycle if it is continuous.
			\item The system (\ref{Ch4s2e_6}) has at most one limit cycle if it is discontinuous.		
		\end{enumerate} 
	\end{theorem}
	\begin{proof}     
		If the system $(\ref{Ch4s2e_6})$ has a periodic solution, then it will meet the line $x=0$ at two distinct points. Suppose that $(0,y_1)$ and $(0,y_2)$ are the points at which the periodic solution of (\ref{Ch4s2e_6}) meet the line $x=0$, where $y_1<y_2$.
		
		Since integrals of the systems $(\ref{Ch4s2e2})$ and $(\ref{Ch4s2e_4})$ are $F(x,y)$ and $H(x,y)$, respectively, we have
		$$F(0,y_1)=F(0,y_2)$$ and $$H(0,y_1)=H(0,y_2).$$
		This implies that, 
		\begin{equation}\label{Ch4s2e_7}
			\left.
			\begin{split} \left(\frac{1}{2}y_1^2-\frac{n}{3}y_1^3\right)-\left(\frac{1}{2}y_2^2-\frac{n}{3}y_2^3\right)=&0,~\text{and}\\	
				-\dfrac{1}{2}\,\delta\,{y_{{1}}}^{2}+\mu\,y_{{1}}+\dfrac{1}{2}\,\delta\,{y_{{2}}}^{2}-
				\mu\,y_{{2}}=&0.
			\end{split}
			\right\}
		\end{equation}
		If the system (\ref{Ch4s2e_6}) is continuous, then the systems $(\ref{Ch4s2e2})$ and $(\ref{Ch4s2e_4})$ must agree on the line $x=0$. Hence, we  have  $\beta=\gamma=0, \delta=1$ and $n=\mu=0.$     
		
		Since $y_1<y_2$,  solutions $(y_1,y_2)$ of (\ref{Ch4s2e_7}) are given by,  
		\begin{align} \label{Ch4s2e_9}
			y_1=-y_2.
		\end{align}
		
		Thus, a continuous PDS consisting of a quadratic center and a linear saddle has a continuous band of periodic solutions and hence no limit cycle.
		
		Next, assume that the system (\ref{Ch4s2e_6})  is not continuous along the boundary $x=0$. Then from (\ref{Ch4s2e_7}) we get,
		\begin{equation}\label{Ch4s2e_10}
			\left.
			\begin{split}
				\frac{1}{2}(y_1-y_2)(y_1+y_2)-\frac{n}{3}(y_1-y_2)(y_1^2+y_1y_2+y_2^2)=&0,\\
				-\frac{1}{2}\delta (y_1-y_2)(y_1+y_2)+\mu(y_1-y_2)=&0.
			\end{split}
			\right\}
		\end{equation}
		Since $y_1<y_2$, from (\ref{Ch4s2e_10}), we have
		\begin{equation}\label{Ch4s2e_11}
			\left.
			\begin{split}
				\frac{1}{2}(y_1+y_2)-\frac{n}{3}(y_1^2+y_1y_2+y_2^2)=&0,\\
				-\frac{1}{2}\delta (y_1+y_2)+\mu=&0.
			\end{split}
			\right\}
		\end{equation}
		Let $e_1=y_1+y_2$ and $e_2=y_1y_2$. Then, from (\ref{Ch4s2e_11}) we get, 
		\begin{equation}
			\left.
			\begin{split}
				e_1=&\dfrac{2\mu}{\delta},\\
				e_2=&-\frac{3\mu}{n\delta}+\dfrac{4\mu^2}{\delta^2}.
			\end{split}
			\right\}	
		\end{equation}
		That is, $y_1, y_2$ satisfy the quadratic equation 
		\begin{align}\label{Ch4s2e_13}
			y^2-\dfrac{2\mu}{\delta} y-\frac{3\mu}{n\delta}+\dfrac{4\mu^2}{\delta^2}=0.
		\end{align}
		The discriminant of the equation (\ref{Ch4s2e_13}) is given by $\Delta= \dfrac{12\mu}{m\delta^2}(\delta-n\mu).$
		Thus, the differential system (\ref{Ch4s2e_6}) has exactly one periodic solution if $\dfrac{12\mu}{m\delta^2}(\delta-n\mu)>0$, otherwise, no solution.	
	\end{proof}

	Now, we consider a PDS in three zones formed by two linear saddles and one nonlinear center that are separated by two straight lines, $x=1$ and $x=-1$.
	We have the following three possibilities:
	\begin{enumerate}
		\item Linear saddles in the regions $\{(x,y):-1<x<1\}$ and  $\{(x,y):x>1\}$, whereas quadratic center in the region $\{(x,y):x<-1\}$.
		
		\item Quardatic center in the region $\{(x,y):-1<x<1\}$, whereas saddles (linear) in the regions $\{(x,y):x<-1\}$ and $\{(x,y):x>1\}$.
		
		\item Linear saddles in the regions $\{(x,y):x<-1\}$ and $\{(x,y):-1<x<1\}$, whereas quadratic center in the region $\{(x,y):x>1\}$.
	\end{enumerate}  
	Here we discuss the case (1). Case (2) and Case (3) can be treated similarly.
	
	Consider a linear system with a Hamiltonian saddle,      
	\begin{align} \label{Ch4s2e_14}
		\dot{X}=\begin{cases}
			-\beta_1 x-\delta_1 y+\mu_1 \\
			\alpha_1 x+\beta_1 y+\gamma_1.
		\end{cases} 
	\end{align} with $\alpha_1=0~\text{or}~~1$; $\alpha_1=0$ imply $\gamma_1=0, \beta_1\neq 0;$ $\alpha_1=1$ imply $\delta_1<\beta_1^2$ and
	its first integral or Hamiltonian is given by
	\begin{align}\label{Ch4s2e_15}
		H_1(x,y)=&-\gamma_1x+\mu_1y-\beta_1xy-\dfrac{1}{2}(\alpha_1{x}^{2}+\delta_1\,{y}^{2}).
	\end{align}    
	
	Let us assume another linear system with a Hamiltonian saddle,      
	\begin{align} \label{Ch4s2e_16}
		\dot{X}=\begin{cases}
			-\beta_2 x-\delta_2 y+\mu_2 \\
			\alpha_2 x+\beta_2 y+\gamma_2
		\end{cases}
	\end{align}
	with  $\alpha_2=0~\text{or}~~1$; $\alpha_2=0$ imply $\gamma_2=0, \beta_2\neq 0;$ $\alpha_2=1$ imply $\delta_2<\beta_2^2$ and Hamiltonian 
	\begin{align}\label{Ch4s2e_17}
		H_2(x,y)=&-\gamma_2x+\mu_2y-\beta_2xy-\dfrac{1}{2}(\alpha_2{x}^{2}+\delta_2{y}^{2}).
	\end{align}       
	Now, define a PDS,   
	\begin{align}\label{Ch4s2e_18}
		\dot{X}=\begin{cases}
			(F_y(x+1,y),-F_x(x+1,y)),&\mbox{if}~x<-1\\
			({H_1}_y(x,y),-{H_1}_x(x,y)),&\mbox{if}~-1<x<1\\
			({H_2}_y(x,y),-{H_2}_x(x,y)),&\mbox{if}~x>1,
		\end{cases}
	\end{align} where $F, H_1$ and $H_2$ are given gy (\ref{Ch4s2e_2.3}), (\ref{Ch4s2e_15}) and (\ref{Ch4s2e_17}), respectively.
	\begin{theorem}\label{thmCh4s3}
		Consider the piecewise differential system (\ref{Ch4s2e_18}). Then we have the following:
		\begin{enumerate}
			\item Systems (\ref{Ch4s2e_18}) has at most four limit cycles if $\delta_1\delta_2n\neq0.$
			\item Systems (\ref{Ch4s2e_18}) has at most two limit cycles if $\delta_1=0$ and $\delta_2n\neq0.$
			\item Systems (\ref{Ch4s2e_18}) has a period annulus if $$(\mu_1+\beta_1)(2\delta_2\gamma_1+(\mu_2-\beta_2)(\mu_1-\beta_1))=0~\text{and}~ \delta_1=n=0.$$
		\end{enumerate}
		
	\end{theorem}
	\begin{proof}  
		
		If the system (\ref{Ch4s2e_18}) has a solution which is a limit cycle, then it meets the lines $x=-1$ and $x=1$ in four different points. Let $(-1,y_1),(-1,y_2),(1,y_3)$ and $(1,y_4)$ be the points where the limit cycle intersects the straight line $x+1=0$ and $x-1=0$, where $y_1<y_2$ and $y_4<y_3$.
		
		Note that, the points $(-1,y_1)$ and $(-1,y_2)$ lie on the same level curve of $F(x+1,y)$, the points $(-1,y_1)$ and $(1,y_4)$ on same level curve of $H_1$ and the points $(1,y_3)$ and $(1,y_4)$ lie on same a level curve of  $H_2$. Hence,
		\begin{equation}
			\left.
			\begin{split}\label{Ch4s2e19}
				F(-1+1,y_1)=F(-1+1,y_2),\\
				H_1(-1,y_1)=H_1(1,y_4),\\
				H_1(-1,y_2)=H_1(1,y_3),\\
				H_2(1,y_3)=H_2(1,y_4).
			\end{split}
			\right\} 
		\end{equation}   
		From (\ref{Ch4s2e19}) we get the following set of equations,
		\begin{align}
			\frac{1}{2}(y_1^2-y_2^2)-\frac{n}{3}(y_1^3-y_2^3)=&0,
			\\
			-\frac{1}{2}\delta_1(y_1^2-y_4^2)+(\mu_1+\beta_1)y_1+(\beta_1-\mu_1)y_4=&-2\gamma_1,\\
			-\frac{1}{2}\delta_1(y_2^2+y_3^2)+(\mu_1+\beta_1)y_2+(\beta_1-\mu_1)y_3=&-2\gamma_1,\\
			-\frac{1}{2}(y_3-y_4)\left(\delta_2(y_3+y_4)-2\mu_2+2\beta_2\right)=&0.
		\end{align}	
		Since $y_1\neq y_2$ and $y_3\neq y_4$, we obtain
		\begin{align}\label{Ch4s3eq1}
			-n(y_1^2+y_1y_2+y_2^2)+\frac{3}{2}(y_1+y_2)&=0,
			\\\label{Ch4s3eq2}
			\delta_1(y_1^2-y_4^2)+2l_1y_1+2m_1y_4+4\gamma_1&=0,\\
			\label{Ch4s3eq3}
			\delta_1(y_2^2-y_3^2)+2l_1y_2+2m_1y_3+4\gamma_1&=0,\\
			\label{Ch4s3eq4}
			\delta_2(y_3+y_4)+2l_2&=0,
		\end{align}
		where $$l_1=-(\mu_1+\beta_1), m_1=\mu_1-\beta_1~\text{and}~l_2=-(\mu_2-\beta_2) .$$
		Using the resultant theory, for instance, see \cite{stiller1996introduction} and \cite{cox1997ideals} p.115, we eliminate $y_4$ from the equations (\ref{Ch4s3eq2}) and (\ref{Ch4s3eq4}) to get,
		\begin{align}\label{Ch4s3eq5}
			\delta_2^2\left( \delta_1(y_1^2-y_3^2)+2l_1y_1-2m_1y_3+4\gamma_1\right)-4\delta_2l_2(\delta_1y_3+m_1)-4\delta_1l_2^2=0.
		\end{align}
		If $\delta_1\neq 0, \delta_2\neq 0$, then eliminating $y_3$ from (\ref{Ch4s3eq3}) and (\ref{Ch4s3eq5}), we get
		\begin{align}\label{Ch4s3eq6}
			Ay_1^2+By_2^2+Cy_1+D y_2+E=0,
		\end{align}
		where \begin{align*}
			A&=-4\delta_1\delta_2^2l_1^2,~~
			B=4\delta_1\delta_2^2m_1^2,~~
			C=-8\delta_2l_1(2\delta_1\delta_2\gamma_1-2\delta_1l_2m_1-\delta_2m_1^2),~~
			D=8\delta_2^2m_1^2l_1\\
			\text{and}~E&=-16\delta_2l_2m_1^3+(-16\delta_1l_2^2+32\delta_2^2\gamma_1)m_1^2+32\delta_1
			\delta_2\gamma_1l_2m_1-16\delta_1{\delta_2}^{2}{\gamma_1}^{2}.
		\end{align*}
		
		Further, for ease of calculations, we can merge the parameters $\delta_1$ and $\delta_2$ with other parameters. Thus, we take parameter values $\delta_1=\delta_2=1$, which will not affect the degree of the polynomial in (\ref{Ch4s3eq7}) as $A_1$ remains nonzero.
		Then we have, 
		\begin{align*}
			A&=-4l_1^2,~~
			B=4m_1^2,~~
			C=8l_1(m_1^2+2l_2m_1-2\gamma_1),~~
			D=8m_1^2l_1~\text{and}\\
			E&=-16l_2m_1^3+(-16l_2^2+16\gamma_1)m_1^2+32
			\gamma_1l_2m_1-16\gamma_1^2.
		\end{align*}
		Next, if $n\neq 0$ then for ease of calculations, we can take $n=1$. Note that this does not affect the degree of the polynomial in (\ref{Ch4s3eq1}) and hence the degree of the polynomial in (\ref{Ch4s3eq7}). Again, by the resultant theory, eliminating $y_2$ from equations (\ref{Ch4s3eq6}) and (\ref{Ch4s3eq1}),
		we get
		\begin{align} \label{Ch4s3eq7}
			A_1y_1^4+A_2y_1^3+A_3y_1^2+A_4y_1+A_5=0,
		\end{align}
		where 
		\begin{align*}
			A_1&=A^2-AB+B^2, A_2=2AC-AD+3B^2-BC-BD,\\
			A_3&=C^2+D^2+2AE-BE-CD-\frac{3}{2}AD-3BD+\frac{9}{4}AB+\frac{9}{4}B^2,\\
			A_4&=\frac{3}{2}(D^2-CD)+2CE-DE+\frac{9}{4}B(C-D),~\text{and}\\
			A_5&=E^2+\frac{9}{4}BE-\frac{3}{2}DE.
		\end{align*}
		Thus, we have at most four real values for $y_1$ and hence $y_2, y_3, y_4$ are determined successively from the equations (\ref{Ch4s3eq1}), (\ref{Ch4s3eq2}), (\ref{Ch4s3eq3}) and (\ref{Ch4s3eq4}).
		Thus, if $\delta_1\delta_2 n\neq 0$, then the system (\ref{Ch4s2e_18}) has at most four limit cycles.
		
		Now, assume that $\delta_1=0$. Similar to the above procedure, we can eliminate $y_2, y_3, y_4$ and get the equation
		\begin{align}
			B_1y_1^2+B_2y_1+B_3=0,
		\end{align}
		where $B_1=n\delta_2^2l_1^2, B_2=2\delta_2l_1nB_3$ and $B_3=(2\delta_2\gamma_1-l_2m_1)$.
		
		Thus, if $\delta_1=0$ and $\delta_2n\neq 0$, then the system (\ref{Ch4s2e_18}) has at most two limit cycles.
		
		Finally, if $\delta_1=0, n=0$, then we get an
		infinite number of periodic orbits  if $l_1B_3=0,$ otherwise no periodic orbit.
	\end{proof}
	\subsection{Quadratic global center and linear saddle}
	Consider the planar piecewise system with switching boundary $x=0$,
	\begin{align}\label{Ch4s3eq17}
		\dot{X}=\begin{cases}
			(y-2x^2-\xi, -2xy),~~&\text{if}~~x<0\\
			(-\beta x-\delta y+\mu, \alpha x+\beta y+\gamma ),~~&\text{if}~~x>0.
		\end{cases}
	\end{align}
	The vector field on the line $x=0$ is defined according to the Filippov convex combination.
	The first integral of the left subsystem is $I(x,y)=y^{-2}\left(x^2-y+\dfrac{\xi}{2}\right)$ and right subsystem is Hamiltonian with Hamiltonian $H(x,y)=-\gamma x+\mu y-\beta xy-\frac{1}{2}(\alpha x^2+\delta y^2)$.
	\begin{theorem}
		If $\beta\neq 0$ then the system (\ref{Ch4s3eq17})  cannot be continuous. If $\beta=0$ and the system is continuous then the right subsystem has a saddle at $(0, \xi)$ on the separation boundary and the system has no periodic solution. 
		
		Further, the system (\ref{Ch4s3eq17}) has exactly one limit cycle if and only if $\mu \delta\left(2\xi+\dfrac{\mu}{\delta}\right)>0$ and $\xi\dfrac{\mu}{\delta}<0$.
	\end{theorem}

	\begin{proof}
		Note that, the system (\ref{Ch4s3eq17}) is continuous if and only if both the subsystems coincide on the line $x=0$ when extended to $x=0$.
		Therefore,   the system  (\ref{Ch4s3eq17}) is continuous if and only if 
		\begin{align}
			(y-\xi, 0)=(-\delta y+\mu, \beta y+\gamma)~~\text{for all}~~y.
		\end{align}
		That is, $\delta=-1,\mu=-\xi, \beta=0=\gamma.$	
		If $\beta\neq 0$ then 
		the system (\ref{Ch4s3eq17}) cannot be continuous.
		If $\beta=0$ and  the system (\ref{Ch4s3eq17})  is continuous then it becomes 
		\begin{align*}
			\dot{X}=\begin{cases}
				(y-2x^2-\xi, -2xy)~~&\text{if}~~x<0\\
				(y-\xi, \alpha x )~~&\text{if}~~x>0
			\end{cases}.
		\end{align*}
		Then the right subsystem has saddle at $(0,\xi)$ provided $\alpha>0$ and due to the saddle separatrices of the right subsystem, the system (\ref{Ch4s3eq17}) has no periodic solution.\\
		Further, the system (\ref{Ch4s3eq17})  has a periodic solution passing through the points $(0,y_1), (0,y_2)$ with $y_1<0<y_2$ if and only if 
		\begin{align*}
			I(0,y_1)&=I(0, y_2) ~~\text{and}~~H(0, y_1)=H(0,y_2).
		\end{align*}
		That is,
		\begin{align}\label{Ch4s3e_3.3}
			\frac{1}{y_1^2}\left(-y_1+\frac{\xi}{2}\right)=& \frac{1}{y_2^2}\left(-y_2+\frac{\xi}{2}\right), ~~\text{and}\\ \label{Ch4s3e_3.4}
			\frac{1}{2}\delta(y_1^2-y_2^2)-\mu(y_1-y_2)=&0.
		\end{align}
		Since $y_1\neq y_2$, from equations (\ref{Ch4s3e_3.3}) and (\ref{Ch4s3e_3.4}), we have
		\begin{equation}\label{Ch4s3e_3,5}
			\left.
			\begin{split}
				y_1+y_2=-2\frac{\mu}{\delta},\\
				\xi(y_1+y_2)-2y_1y_2=0.
			\end{split}
			\right\}
		\end{equation}
		From  (\ref{Ch4s3e_3,5}), we have $$y_1+y_2=-2\frac{\mu}{\delta},~\text{and}~ y_1y_2=-2\xi\frac{\mu}{\delta}.$$
		
		Thus, $y_1$ and $y_2$ are roots of the quadratic equation \begin{align} \label{Ch4s3e_3.6}
			y^2+2\dfrac{\mu}{\delta}y-2\xi\dfrac{\mu}{\delta}=0.
		\end{align}
		Equation (\ref{Ch4s3e_3.6}) has two distinct roots with opposite signs if  and only if
		$\mu \delta\left(2\xi+\dfrac{\mu}{\delta}\right)>0$ and $\xi\dfrac{\mu}{\delta}<0$.
		
		Therefore, the system (\ref{Ch4s3eq17}) has exactly one limit cycle if and only if $\mu \delta\left(2\xi+\dfrac{\mu}{\delta}\right)>0$ and $\xi\dfrac{\mu}{\delta}<0$. Otherwise, the system has no limit cycle.
	\end{proof}
	
	Now, consider the planar piecewise system placed in three zones with switching boundaries $x=\pm 1$ and formed by a global center and two linear saddles,
	\begin{align}\label{Ch4s3eq_18}
		\dot{X}=\begin{cases}
			(y-2(x+1)^2-\xi, -2(x+1)y),~~&\text{if}~~x<-1\\
			(-\beta_1 x-\delta_1 y+\mu_1, \alpha_1 x+\beta_1 y+\gamma_1 ),~~&\text{if}~~-1<x<1\\
			(-\beta_2 x-\delta_2 y+\mu_2, \alpha_2 x+\beta_2 y+\gamma_2 ),~~&\text{if}~~x>1.
		\end{cases}
	\end{align}
	Note that, the vector fields on the lines $x=\pm1$ are defined according to the Filippov convex combination.
	\begin{theorem} 	Consider the piecewise differential system (\ref{Ch4s3eq_18}) with $\delta_1\delta_2\neq 0.$ Let $l_1=-(\beta_1+\mu_1)$.  Then we have the following:
		\begin{enumerate}
			\item System (\ref{Ch4s3eq_18}) has at most four limit cycles when $l_1\xi\neq 0$.
			\item System (\ref{Ch4s3eq_18}) has at most three limit cycles when $l_1\neq 0, \xi=0$.
			\item	System (\ref{Ch4s3eq_18}) has at most two limit cycles when $l_1=0, \xi\neq 0$.
			\item	System (\ref{Ch4s3eq_18}) has no periodic solution when $l_1=\xi= 0$.
		\end{enumerate}
	\end{theorem}
	\begin{proof}
		Assume that the system has a periodic solution passing through four points namely,
		
		$(-1,y_1), (-1,y_2), (1, y_3)$ and $ (1, y_4)$ with $y_1>y_2$ and $y_4>y_3$.
		
		By similar argument as in Theorem \ref{thmCh4s3}, we get the following set of equations in $y_1,y_2,y_3$ and $y_4$,
		\begin{align}\label{Ch4s4e_4.2}
			\xi(y_1+y_2)-2y_1y_2=&0,\\ \label{Ch4s4e_4.3}
			\delta_1(y_3^2-y_2^2)+2l_1y_2+2m_1y_3+4\gamma_1=&0,\\ \label{Ch4s4e_4.4}
			\delta_1(y_4^2-y_1^2)+2l_1y_1+2m_1y_4+4\gamma_1=&0,\\ \label{Ch4s4e_4.5}
			\delta_2(y_3+y_4)+2l_2=&0,
		\end{align}
		where $l_1=-(\beta_1+\mu_1), m_1=\beta_1-\mu_1$ and $l_2=\beta_2-\mu_2$.
		
		Assume that $\delta_1\delta_2n\neq0$. From resultant theory, for instance see \cite{cox1997ideals,stiller1996introduction}, elimination of $y_3, y_4$ from the equations (\ref{Ch4s4e_4.3}), (\ref{Ch4s4e_4.4}) and (\ref{Ch4s4e_4.5}) gives,
		\begin{align}\label{Ch4s4e_4.6}
			Ay_1^2+By_2^2+Cy_1+Dy_2+E=0,
		\end{align}
		\begin{align*}
			\text{where}~A&=-4\delta_1\delta_2^2l_1^2,~
			B=4\delta_1\delta_2^2m_1^2,~ C=-8\delta_2l_1(2\delta_1\delta_2\gamma_1-2\delta_1l_2m_1-\delta_2m_1^2),~D=8\delta_2^2m_1^2l_1,\\
			\text{and}~E&=-16\delta_2l_2m_1^3+(-16\delta_1l_2^2+32\delta_2^2\gamma_1)m_1^2+32\delta_1
			\delta_2\gamma_1l_2m_1-16\delta_1{\delta_2}^{2}{\gamma_1}^{2}.
		\end{align*}
		Now, for ease of calculations if we put $\delta_1=\delta_2=1,$ then
		\begin{align*}
			A&=-4l_1^2,~
			B=4m_1^2,~
			C=8l_1(m_1^2+2l_2m_1-2\gamma_1),~
			D=8m_1^2l_1~\text{and}\\
			E&=-16l_2m_1^3+(-16l_2^2+16\gamma_1)m_1^2+32
			\gamma_1l_2m_1-16\gamma_1^2.
		\end{align*}
		Again, eliminating $y_2$ from equations (\ref{Ch4s4e_4.2}) and (\ref{Ch4s4e_4.6})
		we get,
		\begin{align}\label{Ch4s4e_4.7}
			&A{y_{{1}}}^{4}+ \left( -2\,A\xi+C \right) {y_{{1}}}^{3}+ \left( A{\xi}
			^{2}+B{\xi}^{2}-2\,C\xi+ D \xi+E \right) {y_{{1}}}^{2} \nonumber\\&+
			\left( C{\xi}^{2}- D {\xi}^{2}-2\,E\xi \right) y_{{1}}
			+E{\xi}^{2}=0.
		\end{align}
		From (\ref{Ch4s4e_4.7}), it is clear that, if $l_1\xi\neq 0$ then system (\ref{Ch4s3eq_18}) admits at most four limit cycles;
		if  $l_1\neq 0, \xi=0$,  then (\ref{Ch4s3eq_18}) admits at most three limit cycles;
		if $l_1=0, \xi\neq 0$  then (\ref{Ch4s3eq_18}) has at most two limit cycles; and if $l_1=\xi= 0$  then system (\ref{Ch4s3eq_18}) has no periodic solution.
	\end{proof}

	\section{General quadratic center and linear saddle}\label{Ch4sec3}
	In this section, we discuss limit cycles of piecewise differential systems formed by a quadratic center and linear saddle.
	Consider a nonlinear system
	\begin{align}\label{Ch4s2e1}
		\dot{X}=&\begin{cases}
			-ax-by+3qy^2+rx^2+2sxy, & \text{if}~x<0\\
			x+ay-3px^2-2rxy-sy^2,& \text{if}~x>0,
		\end{cases}
	\end{align}
	where $a,b,p,q,r,s$ denote the real constants. Assume that this system has a center at the origin. Note that, the system is Hamiltonian.
	
	The first integral (Hamiltonian) of the above system is given by, 
	\begin{align}\label{Ch4s3e_3.2}
		G(x,y)=&p{x}^{3}+q{y}^{3}+r{x}^{2}y+sx{y
		}^{2}-\dfrac{1}{2}{x}^{2}-axy-\dfrac{1}{2}b{y}^{2}.
	\end{align}
	Note that, if $G$ has strict local maxima or local minima at the origin, then the Hamiltonian system (\ref{Ch4s2e1}) has a center at the origin, for instance, see lemma in \cite{perko2013differential}, p.172. Hence, the Hamiltonian system (\ref{Ch4s2e1}) has a center at the origin if and only if $-a^2+b>0$
	
	\subsection{Nonlinear center and saddle separated by a straight line}
	
	Now, consider a piecewise smooth differential system separated by a straight line and formed by a nonlinear center and a linear Hamiltonian saddle;
	\begin{align}\label{Ch4s2e6}
		\dot{X}=\begin{cases}
			(G_y(x,y),-G_x(x,y))&~\text{if}~x<0\\
			(H_y(x,y),-H_x(x,y))&~\text{if}~x>0,
		\end{cases}
	\end{align} where $G$ is given by (\ref{Ch4s3e_3.2}) and $H$ is given by (\ref{Ch4s2e_2.6}).
	\begin{theorem}
		Consider the piecewise smooth system (\ref{Ch4s2e6}). 
		\begin{enumerate}
			\item If the system is continuous, then it has no limit cycle.
			\item If the system is discontinuous, then it can have at most two limit cycles when		
			$$2b\delta\,\mu\,q-3\,{\mu}^{2}{q}^{2}\geq 0.$$
		\end{enumerate} 
	\end{theorem}
	\begin{proof}     
		Note that, the system $(\ref{Ch4s2e6})$ has a periodic orbit being a limit cycle, then it must intersect the $y$-axis exactly at two points. Let $(0,y_1)$ and $(0,y_2)$ be the points at which the above-mentioned periodic orbit intersects the line $x=0$, where $y_1<y_2$.  Since, $G(x,y)$ and $H(x,y)$ are the first integrals of the systems $(\ref{Ch4s2e1})$ and $(\ref{Ch4s2e_4})$ respectively, we have
		$$G(0,y_1)=G(0,y_2)$$ and $$H(0,y_1)=H(0,y_2).$$
		This implies that, \begin{align}\label{Ch4s2e7} -\dfrac{1}{2}\,b{y_{{2}}}^{2}+q{y_{{2}}}^{3}+\dfrac{1}{2}\,b{y_{{1}}}^{2}-q{y_{{1}}}^{3}=&0
		\end{align}
		and
		\begin{align}\label{Ch4s2e8}
			-\dfrac{1}{2}\,\delta\,{y_{{1}}}^{2}+\mu\,y_{{1}}+\dfrac{1}{2}\,\delta\,{y_{{2}}}^{2}-
			\mu\,y_{{2}}=&0.
		\end{align}
		If the system (\ref{Ch4s2e6}) is continuous, then the systems $(\ref{Ch4s2e1})$ and $(\ref{Ch4s2e_4})$ must coincide at every point on the line $x=0$, so that we  have  $q=0,s=0,b=\delta,\mu=0,a=\beta$ and $0=\gamma.$     
		
		Since, $y_1<y_2$,  solutions $(y_1,y_2)$ of the system (\ref{Ch4s2e7})-(\ref{Ch4s2e8}) are given by the equations,   
		\begin{equation} \label{Ch4s2e9}
			\left.
			\begin{split}
				b(y_1-y_2) (y_1+y_2)=&0\\
				\delta\left( y_{{1}}-y_{{2}} \right)  (y_1+
				y_2)=&0.
			\end{split}
			\right\}
		\end{equation}
		If $b=0$, then $\Delta(0,0)=0-a^2<0$. This contradicts to  fact that system $(\ref{Ch4s2e1})$ has a center at $(0,0)$. Hence, $b\neq 0$.
		From   (\ref{Ch4s2e9}) we get the solutions $(y_1, y_2)$ as,
		
		$$y_1=-y_2.$$
		Thus, the continuous piecewise smooth differential system formed by a nonlinear center and a linear Hamiltonian saddle have a continuum of periodic solutions and hence no limit cycles.
		
		Now, assume that the system (\ref{Ch4s2e6})  is discontinuous along the separation boundary $x=0$. From (\ref{Ch4s2e7}-(\ref{Ch4s2e8})) we get that,
		\begin{align*}
			(y_1-y_2)\left( q(y_1^2+y_1y_2+y_2^2)+\frac{b}{2}(y_1+y_2)\right)=&0,~~\text{and}\\
			(y_1-y_2)\left(\delta(y_1+y_2)-2\mu \right)=&0.
		\end{align*}
		Since $y_1<y_2$,
		\begin{align*}
			q(y_1^2+y_1y_2+y_2^2)+\frac{b}{2}(y_1+y_2)=&0,\\
			\delta(y_1+y_2)-2\mu =&0.
		\end{align*}
		Hence,
		$ y_1=\dfrac{2\mu}{\delta}-y_2~
		\text{with}~y_2={\dfrac {\mu}{\delta}}+\dfrac{\sqrt{\theta}}{q\delta},~\text{or}~~y_2={\dfrac {\mu}{\delta}}-\dfrac{\sqrt{\theta}}{q\delta},$
		where $$\theta=2b\delta\,\mu\,q-3\,{\mu}^{2}{q}^{2}\geq 0$$ and $q\neq 0.$

		Thus, the differential system (\ref{Ch4s2e6}) has at most two limit cycles if $$q(2b\delta\mu-3\,{\mu}^{2}q)\geq 0~~\text{and}~q\neq 0.$$
	\end{proof}
	
	\subsection{Nonlinear center and linear saddle separated by two straight lines} \label{Ch4s3}
	Now, consider a piecewise smooth system in three zones separated by two straight lines and formed by two linear Hamiltonian saddles and one nonlinear center. For simplicity, we consider two lines $x=-1$ and $x=1$ as the separation boundaries.

	Now define a piecewise smooth system   
	\begin{align}\label{Ch4s3e5}
		\dot{X}=\begin{cases}
			(G_y(x+1,y),-G_x(x+1,y))&\mbox{if}~x<-1\\
			({H_1}_y(x,y),-{H_1}_x(x,y))&\mbox{if}~-1<x<1\\
			({H_2}_y(x,y),-{H_2}_x(x,y))&\mbox{if}~x>1,
		\end{cases}
	\end{align} where $G, H_1$ and $H_2$ are given by (\ref{Ch4s3e_3.2}), (\ref{Ch4s2e_15}) and (\ref{Ch4s2e_17}), respectively.
	\begin{theorem}
		Consider the piecewise smooth system (\ref{Ch4s3e5}). Then we have the following:
		\begin{enumerate}
			\item  The system (\ref{Ch4s3e5}) can have at most four limit cycles if $\delta_1\delta_2q\neq 0$.
			\item The system (\ref{Ch4s3e5}) can have at most three limit cycles if $q=0$ and $b\delta_1\delta_2\neq 0$.
			\item If $\delta_1\delta_2\neq 0$ and $b\delta_1-6(\beta_1+\mu_1)q=0$, then the system (\ref{Ch4s3e5}) can have at most two limit cycles.
			\item If $\delta_1=0,\delta_2\neq 0$, then the system (\ref{Ch4s3e5})
			can have at most one limit cycle and exactly one limit cycle when $-4(\beta_1^2-\mu_1^2)q\neq 0$.
			\item The system (\ref{Ch4s3e5}) has a period annulus if $\delta_1=0,-4(\beta_1^2-\mu_1^2)\delta_2=0$ and $8\delta_2\gamma_1+4(\beta_1-\mu_1)(\beta_2-\mu_2)=0$.
		\end{enumerate}
	\end{theorem}
	\begin{proof}  
		Note that, if the system  (\ref{Ch4s3e5}) has a periodic orbit which is a limit cycle, then it intersects the lines $x=-1$ and $x=1$ in exactly four different points viz, $(-1,y_1),(-1,y_2),(1,y_3)$ and $(1,y_4)$, where $y_1>y_2$ and $y_4>y_3$.
		
		Since, $G, H_1$ and $H_2$ are first integrals, we have
		\begin{align*}
			G(-1+1,y_1)=&G(-1+1,y_2),\\
			H_1(-1,y_1)=&H_1(1,y_4),\\
			H_1(-1,y_2)=&H_1(1,y_3),\\
			H_2(1,y_3)=&H_2(1,y_4). 
		\end{align*}  	
		That is, \begin{align}\label{Ch4s2e15}
			2q(y_1^2+y_1y_2+y_2^2)-b(y_1+y_2)=&0,\\ \label{Ch4s2e16}
			\delta_1(y_2^2-{y_{{3}}}^{2})+l_1y_2+m_1y_3+k_1
			=&0,\\ \label{Ch4s2e17}
			\delta_1(y_1^2-{y_{{4}}}^{2})+l_1y_1+m_1y_4+k_1
			=&0,\\ \label{Ch4s2e18}
			\delta_{{2}}(y_3+y_4)+ l_2=&0,
		\end{align} 
		where $l_1=-2(\beta_1+\mu_1), m_1=2(\beta_1-\mu_1), l_2=2(\beta_2-\mu_2)$ and $k_1=4\gamma_1.$ 
		
		Let $u:=y_1+y_2,v:=y_1-y_2, w:=y_3+y_4$ and $z:=y_3-y_4.$ Then the equations (\ref{Ch4s2e15})-(\ref{Ch4s2e18}) becomes,
		\begin{align}\label{Ch4s3e_3.13}
			\dfrac{1}{2}q \left(3{u}^{2}+{v}^{2} \right) -bu=&0,\\ \label{Ch4s3e_3.14}
			\delta_{{1}} \left( uv+wz \right) -l_{{1}}v-m_{{1}}z=&0,\\ \label{Ch4s3e_3.15}
			\dfrac{\delta_{{1}}}{2} \left( {u}^{2}+{v}^{2}-{w}^{2}-{z}^{2
			} \right) +l_{{1}}u-m_{{1}}w+2\,k_{{1}}=&0,\\ \label{Ch4s3e_3.16}
			w\delta_{{2}}+l_{{2}}=&0.	
		\end{align}
		Using resultant techniques to eliminate $w$ from the equations (\ref{Ch4s3e_3.14}), (\ref{Ch4s3e_3.15}) and (\ref{Ch4s3e_3.16}), we get
		\begin{align}\label{Ch4s3e_3.17}
			uv\delta_{{1}}\delta_{{2}}-v\delta_{{2}}l_{{1}}-z\delta_{{1}}l_{{2}}-z
			\delta_{{2}}m_{{1}}=&0\\ \label{Ch4s3e_3.18}
			{u}^{2}\delta_{{1}}{\delta_{{2}}}^{2}+{v}^{2}\delta_{{1}}{\delta_{{2}}
			}^{2}-{z}^{2}\delta_{{1}}{\delta_{{2}}}^{2}+2\,u{\delta_{{2}}}^{2}l_{{
					1}}-\delta_{{1}}{l_{{2}}}^{2}+4\,{\delta_{{2}}}^{2}k_{{1}}+2\,\delta_{
				{2}}l_{{2}}m_{{1}}=&0.
		\end{align}	
		Similarly, eliminating $z$ from (\ref{Ch4s3e_3.17}) and (\ref{Ch4s3e_3.18}) we get
		\begin{align}\label{Ch4s3e_3.19}
			-{u}^{2}{v}^{2}{\delta_{{1}}}^{3}{\delta_{{2}}}^{4}+2\,u{v}^{2}{\delta_{{1}}}^{2}{\delta_{{2}}}^{4}l_{{1}}+A_1u^2+A_2v^2 +A_3u+K=0,
		\end{align}	
		where 
		\begin{align*} 
			A_1=&{\delta_{{1}}}^{3}{\delta_
				{{2}}}^{2}{l_{{2}}}^{2}+2{\delta_{{1}}}^{2}{\delta_{{2}}}^{3}
			l_{{2}}m_{{1}}+\delta_{{1}}{\delta_{{2}}}^{4}{m_{{1}}}^{2}\\
			A_2=&{\delta_{{1}}}^{3}{\delta_{{2}}}^{2}{l_{{2}}}^{2}+2{\delta
				_{{1}}}^{2}{\delta_{{2}}}^{3}l_{{2}}m_{{1}}-\delta_{{1}}{\delta
				_{{2}}}^{4}{l_{{1}}}^{2}+\delta_{{1}}{\delta_{{2}}}^{4}{m_{{1}}
			}^{2}\\
			A_3=&2{\delta_{{1}}}^{2}{\delta_{{2}}}^{2}l_{{1}}{l_{{2}}}^{2}+4\delta_{{1}}{\delta_{{2}}}^{3}l_{{1}}l_{{2}}m_{{1}}+2{\delta_{{2}}
			}^{4}l_{{1}}{m_{{1}}}^{2},~\text{and}\\
			K=&-{\delta_{{1}}}^{3}{l_{{2}}}^{4}+4\,{\delta_{{1}}}^{2}{\delta_{{2}}}
			^{2}k_{{1}}{l_{{2}}}^{2}+8\,\delta_{{1}}{\delta_{{2}}}^{3}k_{{1}}l_{{2
			}}m_{{1}}+3\,\delta_{{1}}{\delta_{{2}}}^{2}{l_{{2}}}^{2}{m_{{1}}}^{2}+
			4\,{\delta_{{2}}}^{4}k_{{1}}{m_{{1}}}^{2}+2\,{\delta_{{2}}}^{3}l_{{2}}
			{m_{{1}}}^{3}.
		\end{align*}	
		Further, by eliminating $v$ from the equations (\ref{Ch4s3e_3.13}) and (\ref{Ch4s3e_3.19}), we get
		\begin{align}\label{Ch4s3e_3.20}
			B_1u^4+B_2 {u}^{3}+B_3 {u}^{2}+ B_4 u+Kq=0,
		\end{align}
		where \begin{align*}
			B_1=&3{\delta_{{1}}}^{3}{\delta_{{2}}}^{4}q,\\
			B_2=& -2{\delta_{
					{1}}}^{3}{\delta_{{2}}}^{4}b-6{\delta_{{1}}}^{2}{\delta_{{2}}}^{4}l_
			{{1}}q,\\
			B_3=&  \left( -2\,{\delta_{{1}}}^{3}{\delta_{{2}}}^{2}{l_{{2}}}^{2}-4\,{
				\delta_{{1}}}^{2}{\delta_{{2}}}^{3}l_{{2}}m_{{1}}+3\,\delta_{{1}}{
				\delta_{{2}}}^{4}{l_{{1}}}^{2}-2\,\delta_{{1}}{\delta_{{2}}}^{4}{m_{{1
			}}}^{2} \right) q+4{\delta_{{1}}}^{2}{\delta_{{2}}}^{4}l_{{1}}b
			,\\
			B_4=&  \left( 2\,{\delta_{{1}}}^{3}{\delta_{{2}}}^{2}{l_{{2}}}^{2}+4\,{
				\delta_{{1}}}^{2}{\delta_{{2}}}^{3}l_{{2}}m_{{1}}-2\,\delta_{{1}}{
				\delta_{{2}}}^{4}{l_{{1}}}^{2}+2\,\delta_{{1}}{\delta_{{2}}}^{4}{m_{{1
			}}}^{2} \right) b\\
			&+ \left( 2\,{\delta_{{1}}}^{2}{\delta_{{2}}}^{2}l_{{1
			}}{l_{{2}}}^{2}+4\,\delta_{{1}}{\delta_{{2}}}^{3}l_{{1}}l_{{2}}m_{{1}}
			+2\,{\delta_{{2}}}^{4}l_{{1}}{m_{{1}}}^{2} \right) q.
		\end{align*}
		From (\ref{Ch4s3e_3.20}), it is clear that the system (\ref{Ch4s3e5}) can have at most four limit cycles. Further, if $q=0$ and $b\delta_1\delta_2\neq 0$, then the equation (\ref{Ch4s3e_3.20}) becomes,
		$$-2\delta_1^3\delta_2^4bu^3+4\delta_1^2\delta_2^4l_1bu^2+C_1u+K=0,$$
		where $C_1=\left( 2\,{\delta_{{1}}}^{3}{\delta_{{2}}}^{2}{l_{{2}}}^{2}+4\,{
			\delta_{{1}}}^{2}{\delta_{{2}}}^{3}l_{{2}}m_{{1}}-2\,\delta_{{1}}{
			\delta_{{2}}}^{4}{l_{{1}}}^{2}+2\,\delta_{{1}}{\delta_{{2}}}^{4}{m_{{1
		}}}^{2} \right) b.$
		
		Thus, the system (\ref{Ch4s3e5}) can have at most three periodic orbits in this case.
		
		Also, if $\delta_1\delta_2\neq 0$ and $b\delta_1+3l_1q=0$, then the system (\ref{Ch4s3e5}) can have at most two periodic orbits.
		
		Suppose that $\delta_1=0$. Then the equation (\ref{Ch4s3e_3.20}) becomes,
		$$2\,{\delta_{{2}}}^{4}l_{{1}}{m_{{1}}}^{2} q u +(4\,{\delta_{{2}}}^{4}k_{{1}}{m_{{1}}}^{2}+2\,{\delta_{{2}}}^{3}l_{{2}}
		{m_{{1}}}^{3})=0.$$
		
		Thus, if $\delta_1=0,\delta_2\neq 0$, then the system (\ref{Ch4s3e5})
		can have at most one limit cycle and exactly one limit cycle when $l_1m_1q\neq 0$.
		
		Finally, if $\delta_1=0,\delta_2l_1m_1=0$ and $2\delta_2k_1+l_2m_1\neq 0$, then the system (\ref{Ch4s3e5}) can not have a limit cycle. Also, there is a period annulus when $2\delta_2k_1+l_2m_1= 0.$
	\end{proof}

	\section{Applications}\label{Ch4sec4}
	First, we present some examples of discontinuous piecewise linear differential systems with one limit cycle from \cite{llibre2022limit} and nonlinear systems with two or more limit cycles.
	\begin{example}
		Consider a piecewise linear system 
		\begin{align}\label{Ch4s4.40}
			\dot{X}=(x',y')=
			\begin{cases}
				(-8-16x-\frac{65}{2}y,-8(x+2y)),&\text{if}~x<-1\\
				(-1-2y,-1-2x), &\text{if}~-1<x<1\\
				(8-8x-10y,8+8x+8y),&\text{if}~x>1.	
			\end{cases}
		\end{align}
	\end{example}
	This system is of the type center-saddle-center.  This discontinuous piecewise differential system has one limit cycle intersecting the two discontinuous straight lines $x=1$ and $x=-1$ at the points $\left(-1,\dfrac{16}{65}+\dfrac{\sqrt{4873}}{36\sqrt{2}}\right), \left(-1,\dfrac{16}{65}-\dfrac{\sqrt{4873}}{36\sqrt{2}}\right),\left(1,-\dfrac{97\sqrt{4873}}{2340\sqrt{2}}\right)$ and $\left(1,\dfrac{97\sqrt{4873}}{2340\sqrt{2}}\right)$ (see Figure (\ref{Ch4Ex4.4.1})).
		\begin{figure}[H]	
			\subfloat[Flow of the system (\ref{Ch4s4.40})]{
				\begin{minipage}[c][1\width]{					0.3\textwidth}   				\centering   				\includegraphics[width=1\textwidth]{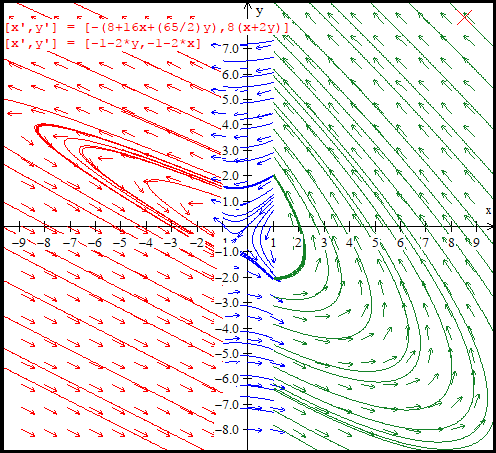} 				\label{Ch4Ex4.4.1a}  		\end{minipage}}  			\qquad \qquad	\qquad \qquad \qquad   			\subfloat[Limit cycle of (\ref{Ch4s4.40})]{  			\begin{minipage}[c][1\width]{   					0.3\textwidth}   				\centering  				\includegraphics[width=1\textwidth]{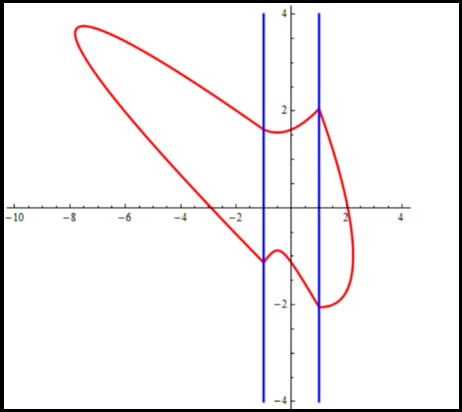}  				\label{Ch4Ex4.4.1b}  		\end{minipage}}     		\caption{Piecewise linear system of the type Center-Saddle-Center (\ref{Ch4s4.40})}  		\label{Ch4Ex4.4.1}   	\end{figure} 

	\begin{example}
		Consider a piecewise linear system 
		\begin{align}\label{Ch4s4.41}
			\dot{X}=(x',y')=
			\begin{cases}
				(-8-16x-\frac{65}{2}y,8(x+2y)),&\text{if}~x<-1\\
				(-1-2y,-1+2x), &\text{if}~-1<x<1\\
				(-y,4-x),&\text{if}~x>1.	
			\end{cases}
		\end{align}
	\end{example}
	The system (\ref{Ch4s4.41}) is a piecewise discontinuous system of the type Center-Center-Saddle. It has one limit cycle intersecting the straight lines $x=-1$ and $x=1$ at the points $(-1,y_1),(-1,y_2), (1,y_3)$ and $(1,y_4)$, where $y_1>y_2$ and $y_4>y_3$ (see Figure (\ref{Ch4Ex4.4.2})).
	
		\begin{figure}[H]
			\subfloat[Flow of the system (\ref{Ch4s4.41})]{   			\begin{minipage}[c][1\width]{  0.3\textwidth}  				\centering   				\includegraphics[width=1\textwidth]{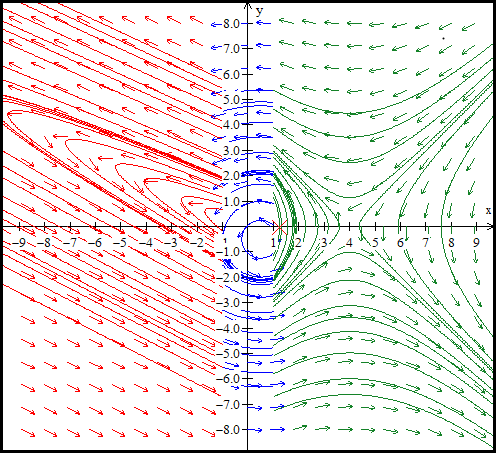}   				\label{Ch4Ex4.4.2a} 
			\end{minipage}} 
			\qquad \qquad	\qquad \qquad \qquad            
			\subfloat[Limit cycle of (\ref{Ch4s4.41})]{  			\begin{minipage}[c][1\width]{   					0.28\textwidth}   		\centering   				\includegraphics[width=1\textwidth]{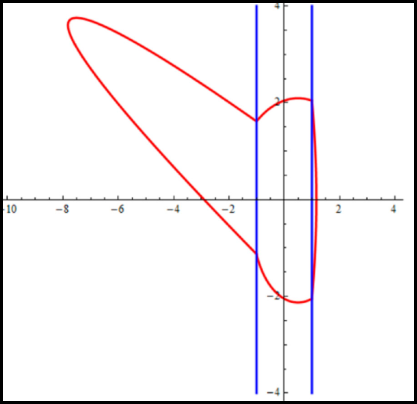}  				\label{Ch4Ex4.4.2b}   		\end{minipage}}   		\caption{Piecewise linear system of the type Center-Center-Saddle (\ref{Ch4s4.41})}   		\label{Ch4Ex4.4.2}
		\end{figure}
	\begin{example}
		Consider a piecewise linear system 
		\begin{align}\label{Ch4s4.42}
			\dot{X}=(x',y')=
			\begin{cases}
				(u,v),&\text{if}~x<-1\\
				(-1-2y,-1+2x), &\text{if}~-1<x<1\\
				(-2y,8-2x),&\text{if}~x>1,	
			\end{cases}
		\end{align}
		where $u=2-\frac{1}{48}(309-\sqrt{157881})x-\frac{65}{786}(405-\sqrt{157881})y$ and $v=-16-2x+\frac{1}{48}(309-\sqrt{157881})y.$
	\end{example}
	The system (\ref{Ch4s4.42}) is a piecewise discontinuous system of the type Saddle-Center-Saddle. It has one limit cycle intersecting the straight lines $x=-1$ and $x=1$ at the points $(-1,y_1),(-1,y_2), (1,y_3)$ and $(1,y_4)$, where $y_1>y_2$ and $y_4>y_3$ (see Figure (\ref{Ch4Ex4.4.3})).
	
			\begin{figure}[H]      		\subfloat[Flow of the system (\ref{Ch4s4.42})]{   			\begin{minipage}[c][1\width]{  	0.3\textwidth}  				\centering   				\includegraphics[width=1\textwidth]{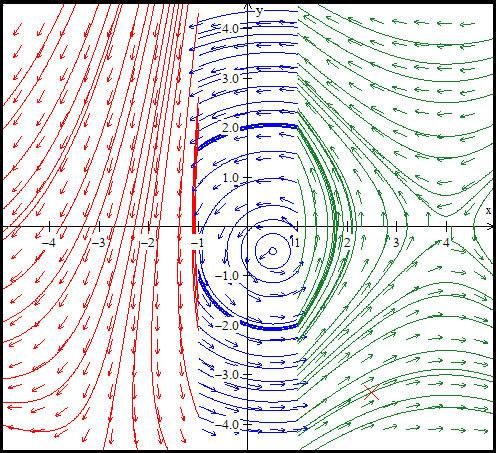}  				\label{Ch4Ex4.4.3a}  		\end{minipage}}   			\qquad \qquad	\qquad \qquad \qquad     		\subfloat[Limit cycle of (\ref{Ch4s4.42})]{  			\begin{minipage}[c][1\width]{  					0.28\textwidth}  \centering  				\includegraphics[width=1\textwidth]{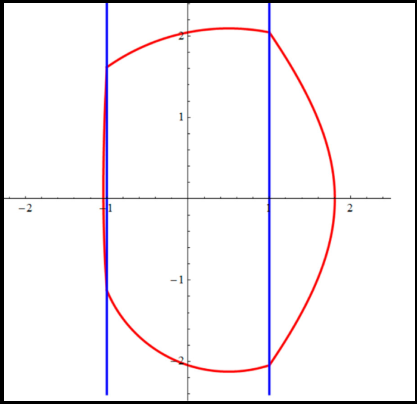}  			\label{Ch4Ex4.4.3b}  		\end{minipage}}  		\caption{Piecewise linear system of the type Saddle-Center-Saddle (\ref{Ch4s4.42})}   		\label{Ch4Ex4.4.3}  	\end{figure} 
	
	\begin{example}
		Consider a piecewise linear system 
		\begin{align}\label{Ch4s4.43}
			\dot{X}=(x',y')=
			\begin{cases}
				(u,v),&\text{if}~x<-1\\
				(-1-2y,-1-2x), &\text{if}~-1<x<1\\
				(8-8x-10y,8+8x+8y),&\text{if}~x>1,	
			\end{cases}
		\end{align}
		where \begin{align*} u=&2-\frac{1}{48}(309-\sqrt{157881})x-\frac{65}{786}(405-\sqrt{157881})y,~\text{and}\\ v=&-16-2x+\frac{1}{48}(309-\sqrt{157881})y.
		\end{align*}
	\end{example}
	The system (\ref{Ch4s4.43}) is a piecewise discontinuous system of the type Saddle-Saddle-Center. It has one limit cycle intersecting the straight lines $x=-1$ and $x=1$ at the points $(-1,y_1),(-1,y_2), (1,y_3)$ and $(1,y_4)$, where $y_1>y_2$ and $y_4>y_3$ (see Figure (\ref{Ch4Ex4.4.4})).
	
		\begin{figure}[H]  	
			\subfloat[Flow of the system (\ref{Ch4s4.43})]{ 
				\begin{minipage}[c][1\width]{  0.30\textwidth} 
					\centering  				\includegraphics[width=1\textwidth]{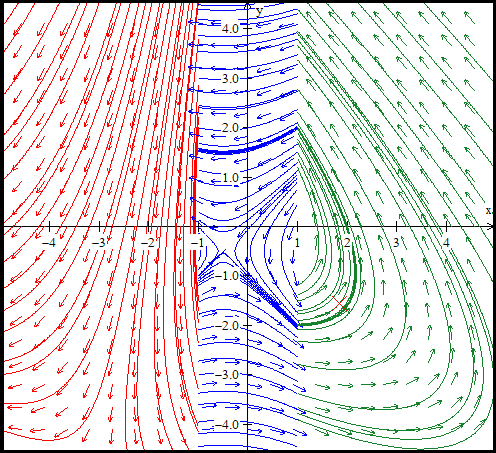} \label{Ch4Ex4.4.4a}  		\end{minipage}}  
			\qquad \qquad	\qquad \qquad \qquad   
			\subfloat[Limit cycle of (\ref{Ch4s4.43})]{ 
				\begin{minipage}[c][1\width]{0.29\textwidth} 
					\centering  				\includegraphics[width=1\textwidth]{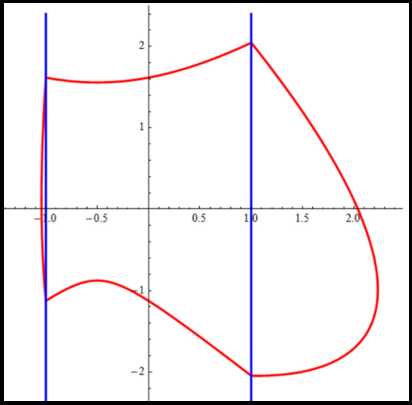}  				\label{Ch4Ex4.4.4b} 
			\end{minipage}} 
			\caption{Piecewise linear system of the type Saddle-Center-Saddle  (\ref{Ch4s4.43})} 
			\label{Ch4Ex4.4.4}
		\end{figure}
	
	Now, we provide examples of nonlinear systems (\ref{Ch4s2e6}) with one limit cycle and that with two limit cycles.
	\begin{example} Consider the differential system 
		\begin{align}\label{Ch4s4e1}
			\dot{X}=\begin{cases}
				(3y^2-4y, x),&~\text{if}~~x<0\\
				(-y+\frac{1}{2}, -x-c),&~\text{if}~~x>0.
			\end{cases}
		\end{align}
	\end{example}
	Note that, integral of the left subsystem is $$F(x,y)=y^3-2y^2-\dfrac{x^2}{2}$$ and integral of the right subsystem is $$H(x,y)=-\dfrac{y^2}{2}+\dfrac{y}{2}+\dfrac{x^2}{2}+cx.$$
	Assume that, the system (\ref{Ch4s4e1}) has a limit cycle passing through the points $(0,y_1)$ and $(0,y_2)$ with $y_1<0<y_2$.
	Then, we have $F(0,y_1)=F(0,y_2)$ and $H(0,y_1)=H(0,y_2)$.
	That is,
	\begin{equation}\label{Ch4s4e4.6}
		\left.
		\begin{split}
			(y_1^2+y_1y_2+y_2^2)-2(y_1+y_2)=&0,\\
			y_1+y_2=&1.	
		\end{split}
		\right\}
	\end{equation}
	Solutions of (\ref{Ch4s4e4.6}) are 
	$$ \left(\frac{1-\sqrt{5}}{2}, \frac{1+\sqrt{5}}{2}\right),~\text{and}~\left(\frac{1+\sqrt{5}}{2}, \frac{1-\sqrt{5}}{2}\right).$$
	Hence, there is exactly one limit cycle that passes through the points (Fig. \ref{Ch4s4fig4.5a}) $$\left(0,y_1=\dfrac{1-\sqrt{5}}{2}\right)~ \text{and} ~\left(0,y_2=\dfrac{1+\sqrt{5}}{2}\right).$$
	
	\begin{example}
		Consider the piecewise differential system
		\begin{align}\label{Ch4s4e4.7}
			\dot{X}=\begin{cases}
				(y-2x^2-0.8,-2xy),&\text{if}~x<0\\
				(x+y-1,-y),&\text{if}~x>0.
			\end{cases}
		\end{align}
		Note that, the Hamiltonian for the left subsystem is $$F(x,y)=\dfrac{1}{y^2}(x^2-y-0.4)$$ and that for the right subsystem is $$H(x,y)=xy+\dfrac{1}{2}y^2-y.$$
		Suppose that the system (\ref{Ch4s4e4.7}) has a periodic solution that passes through the points $(0,y_1)$ and $(0,y_2)$ with $y_1>y_2$. Then, we have $F(0,y_1)=F(0,y_2)$ and $H(0,y_1)=H(0,y_2)$. This gives 
		\begin{equation}\label{Ch4s4e4.8}
			\left.\begin{split}
				y_1y_2-(0.4)(y_1+y_2)=&0,\\
				y_1+y_2=&2.
			\end{split}\right\}
		\end{equation}
		Eliminating $y_2$ from the equations in (\ref{Ch4s4e4.8}), we get 
		\begin{align}\label{Ch4s4e4.9}
			y_1^2-2y_1+\dfrac{4}{5}=0.
		\end{align}
		Roots of (\ref{Ch4s4e4.9}) are $y=\dfrac{5+\sqrt{5}}{5}, \dfrac{5-\sqrt{5}}{5}.$
		
		Thus, there is exactly one limit cycle of (\ref{Ch4s4e4.7}) which passes through the points, 
		$$\left(0,\dfrac{5+\sqrt{5}}{5}\right)~\text{and}~ \left(0,\dfrac{5-\sqrt{5}}{5}\right).$$
	\end{example}
	
	Now, we provide an example of a system  (\ref{Ch4s3e5})  having one, two, and three limit cycles.
	\begin{example}
		Consider a piecewise differential system 
		\begin{align}\label{Ch4s4e2}
			\dot{X}=\begin{cases}
				(-3y^2+2ay, -x-1), & \text{if}~x<-1\\
				(2y+\frac{1}{4}, 2x+\frac{1}{4}),& \text{if}~-1<x<1\\
				(2y, 2x-4),& \text{if}~~x>1.
		\end{cases}\end{align}
	\end{example}
	Note that, the subsystem $\dot{X}=(-3y^2+2ay, -x-1)$ has a center at $(-1,0)$ and a saddle at $\left(-1, \dfrac{2a}{3}\right)$ when $b>0$. Also, note that this system has a homoclinic orbit around the point $(-1,0)$ that passes through $\left(-1, \dfrac{b}{3}\right)$. 
	
	The subsystem \begin{align}\label{Ch4s4e4.11}
		\dot{X}=\left(2y+\dfrac{1}{4},
		2x+\frac{1}{4}\right)
	\end{align} has a saddle at $(-\mu_1, 0)$, whereas the  subsystem \begin{align}\label{Ch4s4e4.12}
		\dot{X}=(2y, 2x-4)
	\end{align} has a saddle at $(-\mu_2,0)$.
	Integrals of the subsystems in (\ref{Ch4s4e2}) are given by
	\begin{align*}
		H(x,y)&=-y^3+ay^2+\frac{1}{2}(x+1)^2,\\
		H_1(x,y)&=y^2+\frac{y}{4}-x^2-\frac{x}{4},\\
		H_2(x,y)&=y^2-x^2+4x,
	\end{align*} respectively.
	
	Suppose that there is a periodic solution of (\ref{Ch4s4e2}) that passes through the points
	$(-1, y_1)$, $(-1, y_2)$, $(1, y_3)$ and $(1,y_4)$ with $y_1\neq y_2$ and $y_3\neq y_4$.
	Then we have
	\begin{align*}
		H(-1,y_1)&=H(-1, y_2),\\ H_1(-1,y_2)&=H_1(1,y_3),\\ H_2(1,y_3)&=H_2(1,y_4),\\ H_1(1,y_4)&=H_1(-1, y_2). 
	\end{align*}
	This implies,
	\begin{align}\label{CH4S4apli4.6eq7}
		y_1^2+y_1y_2+y_2^2-2a(y_1+y_2)&=0,\\ y_2^2-y_3^2+\frac{1}{2}(y_2-y_3)+\frac{1}{2}&=0,\\ y_3+y_4&=0,\\ 
		y_1^2-y_4^2+\frac{1}{2}(y_1-y_4)+\frac{1}{2}&=0.
	\end{align}
	Eliminating $y_3$ from the above equations, we get
	\begin{align}
		y_1^2+y_1y_2+y_2^2-2a(y_1+y_2)&=0,\\ \label{eq4.12} y_2^2-y_4^2+\frac{1}{2}(y_2+y_4)+\frac{1}{2}&=0, \\ \label{eq4.13}
		y_1^2-y_4^2+\frac{1}{2}(y_1-y_4)+\frac{1}{2}&=0.
	\end{align}
	Subtracting the equation (\ref{eq4.13}) from  (\ref{eq4.12}), we get
	\begin{align*}
		y_4=y_1^2-y_2^2+\frac{1}{2}(y_1-y_2).
	\end{align*}
	Substituting  this value of $y_4$ in equation (\ref{eq4.12}) gives,
	\begin{align}\label{Ch4s4e4.15}
		\frac{1}{4}{y_1}^{2}-{y_1}^{4}+2{y_1}^{2}{y_2}^{2}-{y_{{1}
		}}^{3}+{y_{{1}}}^{2}y_{{2}}-{y_{{2}}}^{4}+{y_{{2}}}^{2}y_{{1}}-{y_{{2}
		}}^{3}+\frac{1}{2}y_{{1}}y_{{2}}+\frac{1}{4}\left({y_{{2}}}^{2}+y_{{1}}+y_{{2}
		}+2\right)
		=0.
	\end{align}
	Finally, eliminating $y_2$ from the equations (\ref{Ch4s4e4.15}) and $(\ref{CH4S4apli4.6eq7})$, we get
	\begin{equation}\label{Ch4s4e4.16}
		\left.
		\begin{split}
			&9{y_{{1}}}^{8}+\frac{1}{16} \left( -384\,a+144 \right) {y_{{1}}}^{7}+\frac{1}{16} \left( 256{a}^{2}-96a+132 \right) {y_{{1}}}^{6}\\&+\frac{1}{16} \left( -
			512{a}^{2}-368\,a-12 \right) {y_{{1}}}^{5}+\frac{1}{16} \left( 512{a}^{3
			}+240{a}^{2}-8a-35 \right) {y_{{1}}}^{4}\\&+\frac{1}{16} \left( 16{a}^{2}
			-48a-47 \right) {y_{{1}}}^{3}+\frac{1}{16} \left( -32{a}^{3}+176{a}^{2
			}+66a-1 \right) {y_{{1}}}^{2}\\&+\frac{1}{16} \left( 32{a}^{2}+8a+2
			\right) y_{{1}}-8{a}^{4}-4{a}^{3}+\frac{1}{2}{a}^{2}+\frac{a+1}{4}=0.
		\end{split}
		\right\}
	\end{equation}
	If we choose $a<-0.63$, then there are four roots, three negative and one positive, for the equation (\ref{Ch4s4e4.16}). Other variables are determined from $y_1$. Hence, the system (\ref{Ch4s4e2}) has four limit cycles if $a<-0.63$. 
	
	If we choose $0.4<a$, then there are four roots, three positive and one negative for the equation (\ref{Ch4s4e4.16}). Hence, the system (\ref{Ch4s4e2}) has four limit cycles in this case.
	
	Similarly, for $-0.6<a<0.38$ the system (\ref{Ch4s4e2}) no limit cycles. 
	
	If $-0.38<a<0.4$, then there are two negative and two positive roots. Hence, the system (\ref{Ch4s4e2}) has four limit cycles.

	\section{{Acknowledgement}}
	The authors would like to express sincere gratitude to the reviewers for their valuable suggestions and comments.
	\section{Conflict of interest statement}
	No funding was received for conducting this study. The authors have no financial or proprietary interests in any material discussed in this article.

\section*{Declaration}	
The final version of this article will be published in \em{Journal of Difference Equations and Applications}.


\end{document}